\documentclass[12pt,english]{amsart}
\usepackage[utf8]{inputenc}
\usepackage[a4paper]{geometry}
\usepackage{babel}
\usepackage{verbatim}
\usepackage{amscd}
\usepackage{amstext}
\usepackage{amsthm}
\usepackage{amssymb}
\usepackage{mathdots}
\usepackage{esint}
\usepackage[unicode=true,
 bookmarks=false,
 breaklinks=false,pdfborder={0 0 1},backref=false,colorlinks=false]
 {hyperref}
\usepackage{breakurl}

\makeatletter
\numberwithin{equation}{section}
\theoremstyle{plain}
\newtheorem{thm}{\protect\theoremname}[section]
  \theoremstyle{definition}
  \newtheorem{defn}[thm]{\protect\definitionname}
  \theoremstyle{remark}
  \newtheorem*{rem*}{\protect\remarkname}
  \theoremstyle{definition}
  \newtheorem{example}[thm]{\protect\examplename}
  \theoremstyle{remark}
  \newtheorem{rem}[thm]{\protect\remarkname}
  \theoremstyle{plain}
  \newtheorem{prop}[thm]{\protect\propositionname}
  \theoremstyle{plain}
  \newtheorem{lem}[thm]{\protect\lemmaname}
  \theoremstyle{plain}
  \newtheorem{cor}[thm]{\protect\corollaryname}

\usepackage{calc}
\usepackage{amstext}
\usepackage{amscd}
\usepackage{breakurl}
\usepackage{fancyhdr}
\usepackage{placeins}
\usepackage{babel}
\usepackage{amsfonts}

\theoremstyle{plain}
  \newtheorem*{thma}{Theorem A}
\newtheorem*{thmb}{Theorem B}

\title[Homotopy type of moduli spaces]{Homotopy type of moduli spaces
  of G-Higgs bundles and reducibility of the nilpotent cone}

\author[C. Florentino]{C. Florentino}

\address{Departamento de Matem\'{a}tica, Faculdade de Ci\^{e}ncias, Universidade de Lisboa,  Edf. C6, Campo Grande 1749-016 Lisboa, Portugal}

\email{caflorentino@fc.ul.pt}

\author[P. B. Gothen]{P. B. Gothen}
\address{Centro de
  Matem\'atica da Universidade do Porto,
Faculdade de Ci\^encias da Universidade do Porto,
Rua do Campo Alegre, 4169-007 Porto, Portugal }
\email{pbgothen@fc.up.pt}

\author[A. Nozad]{A. Nozad}

\address{School of Mathematics, Institute for Research in Fundamental Sciences (IPM), P.O. Box: 19395-5746, Tehran, Iran.}

\email{anozad@ipm.ir}

\thanks{ This work was partially supported by CMUP
  (UID/MAT/00144/2013), CMAF-CIO (MAT/04561), the projects PTDC/MAT-GEO/2823/2014 
  and EXCL/MAT-GEO/0222/2012 funded
  by FCT (Portugal) with national funds and a grant from IPM.
  The authors acknow\-ledge support from U.S. National Science
  Foundation grants DMS 1107452, 1107263, 1107367 "RNMS: GEometric
  structures And Representation varieties" (the GEAR Network)}
\subjclass{14H60, 32L05}
\keywords{Higgs bundles, Character varieties, Nilpotent cone, Topology of moduli spaces}

\newcommand{\suchthat}{\;\;|\;\;}
\newcommand{\abs}[1]{\lvert#1\rvert}
\renewcommand{\leq}{\leqslant}

\renewcommand{\geq}{\geqslant}

\newcommand{\R}{\mathbb{R}}

\newcommand{\Z}{\mathbb{Z}}
\renewcommand{\C}{\mathbb{C}}

\newcommand{\lie}{\mathfrak}

\newcommand{\SU}{\mathrm{SU}}
\renewcommand{\U}{\mathrm{U}}
\newcommand{\GL}{\mathrm{GL}}
\newcommand{\SL}{\mathrm{SL}}
\newcommand{\SO}{\mathrm{SO}}
\newcommand{\Sp}{\mathrm{Sp}}

\newcommand{\norm}[1]{\lVert#1\rVert}

\DeclareMathOperator{\Ad}{Ad}

\DeclareMathOperator{\tr}{tr}
\DeclareMathOperator{\rk}{rk}
\DeclareMathOperator{\im}{im}

\DeclareMathOperator{\Hom}{Hom}
\DeclareMathOperator{\End}{End}

  \providecommand{\corollaryname}{Corollary}
  \providecommand{\definitionname}{Definition}
  \providecommand{\examplename}{Example}
  \providecommand{\lemmaname}{Lemma}
  \providecommand{\propositionname}{Proposition}
  \providecommand{\remarkname}{Remark}
\providecommand{\theoremname}{Theorem}

\makeatother

  \providecommand{\corollaryname}{Corollary}
  \providecommand{\definitionname}{Definition}
  \providecommand{\examplename}{Example}
  \providecommand{\lemmaname}{Lemma}
  \providecommand{\propositionname}{Proposition}
  \providecommand{\remarkname}{Remark}
\providecommand{\theoremname}{Theorem}

\begin{document}
\begin{abstract}
Let $G$ be a real reductive Lie group, and $H^{\mathbb{C}}$ the
complexification of its maximal compact subgroup $H\subset G$. We
consider classes of semistable $G$-Higgs bundles over a Riemann surface
$X$ of genus $g\geq2$ whose underlying $H^{\mathbb{C}}$-principal
bundle is unstable. This allows us to find obstructions to a deformation
retract from the moduli space of $G$-Higgs bundles over $X$ to the
moduli space of $H^{\mathbb{C}}$-bundles over $X$, in contrast with
the situation when $g=1$, and to show reducibility of the nilpotent
cone of the moduli space of $G$-Higgs bundles, for $G$ complex. 
\end{abstract}

\maketitle

\section{Introduction}

A \emph{Higgs bundle} on a Riemann surface $X$ is a pair $(E,\varphi)$,
where $E$ is a rank $n$ holomorphic vector bundle over $X$ and
$\varphi\in H^{0}(\End(E)\otimes K)$ is a holomorphic endomorphism
of $E$ twisted by the canonical bundle $K$ of $X$. Higgs bundles
appeared first in the work of Hitchin \cite{Hi87} and Simpson \cite{Si92,Si88}.
The non-abelian Hodge Theorem \cite{Co88,Do87,Hi87,Si88} identifies
the moduli space of Higgs bundles with the \emph{character variety}
for representations of the fundamental group of $X$ into $\GL(n,\mathbb{C})$.

The appropriate objects for extending the non-abelian Hodge Theorem
to representations of the fundamental group in a \emph{real} reductive
Lie group $G$ (see, e.g., \cite{Hi92,GGM09,Go14}) are called \emph{$G$-Higgs
bundles}. There are natural notions of stability, semistability, and
polystability for $G$-Higgs bundles, leading to corresponding moduli
spaces $\mathcal{M}(G)$ (see \cite{GGM09} for the general theory).
Again, there is an identification between $\mathcal{M}(G)$ and the
moduli space of flat $G$-connections on $X$.

Motivated partially by this identification, the moduli space of $G$-Higgs
bundles has been extensively studied. When $G$ is a complex semisimple
Lie group Biswas and Florentino proved in \cite{BF} that the moduli
space of topologically trivial principal $G$-bundles over a compact
Riemann surface (which are actually $H$-Higgs bundles, where $H$
is the maximal compact subgroup of $G$) is not a deformation retraction
of the moduli space of topologically trivial $G$-Higgs bundles. This
result contrasts with the main theorem of Florentino and Lawton \cite{FL}
which says that the moduli space of flat $H$-connections on an \emph{open
surface $X$ }is a strong deformation retraction of the moduli space
of flat $G$-connections on $X$, for complex reductive $G$.

Our aim in this paper is to generalize the above mentioned theorem
of Biswas and Florentino to the case of real reductive Lie groups.
Using the non-abelian Hodge theorem, the question is to prove that
the moduli spaces of semistable principal $H^{\mathbb{C}}$-bundles,
which we denote by $\mathcal{N}(H^{\mathbb{C}})$, is not a deformation
retraction of the moduli spaces of semistable $G$-Higgs bundles $\mathcal{M}(G)$,
where $H^{\C}$ is the complexification of $H$. We recall that the
topological invariants of the underlying principal bundles label unions
of connected components of the moduli spaces, so in order to study
deformation retraction from $\mathcal{M}(G)$ to $\mathcal{N}(H^{\mathbb{C}})$
we should consider separately each topological type. In this paper,
we address the case of trivial topological type.

Our strategy is as follows. We use the $\C^{\ast}$-action on the
moduli space of $G$-Higgs bundles, given by multiplication of the
Higgs field, and show (Proposition \ref{prop:deformation-nil}) that
it provides a deformation retraction onto the \emph{nilpotent cone}:
the pre-image of zero under the Hitchin map, defined in section \ref{sec:nilpotent-cone}.
Therefore, we reduce the question to finding obstructions to a deformation
from the nilpotent cone to $\mathcal{N}(H^{\mathbb{C}})$. Then we
prove that such obstructions are semistable $G$-Higgs bundles whose
underlying $H^{\mathbb{C}}$-bundle is unstable and we show existence
of these obstructions by using the construction of \cite{GPR}, stated
in Proposition \ref{existence-unstable-bundle}. This result allows
us also to deduce the reducibility of the nilpotent cone of the moduli
space of $G$-Higgs bundles when $G$ is a connected reductive complex
Lie group.

More precisely, our main results are the following theorems (see Theorems
\ref{thm:reducibility-nilpotent} and \ref{thm:homotopy-type} below;
note that the moduli spaces may be singular).

\begin{thma} Let $G$ be a non-abelian connected reductive complex
Lie group. Then the nilpotent cone in the moduli space of $G$-Higgs
bundles of trivial topological type is not irreducible. \end{thma}

\begin{thmb} Let $G$ be a non-abelian (real or complex) connected
reductive Lie group of non-Hermitian type or connected simple real
Lie group of Hermitian non-tube type. Then the moduli space of semistable
principal $H^{\C}$-bundles of trivial topological type is not a deformation
retraction of the moduli space of semistable $G$-Higgs bundles of
trivial topological type. \end{thmb}

\section{Moduli of Higgs bundles and the nilpotent cone}

\subsection{$G$-Higgs bundles}

Let $X$ be a compact connected Riemann surface of genus $g$, for
$g\geq2$, and let $K=T^{*}X$ be the canonical bundle of $X$. Let
$G$ be a (real or complex) connected reductive Lie group with a choice
of a maximal compact subgroup $H\subset G$, and denote by $H^{\mathbb{C}}$
the complexification of $H$.


By an $H^{\mathbb{C}}$\emph{-bundle over $X$} we always mean a \emph{holomorphic
}principal $H^{\mathbb{C}}$-bundle over $X$. Recall that this is
a holomorphic fibre bundle $\pi:E\to X$ with a holomorphic $H^{\mathbb{C}}$-action
which is free and transitive on each fibre and $E$ is required to
admit holomorphic $H^{\mathbb{C}}$-equivariant local trivializations
$E|_{U}\cong U\times H^{\mathbb{C}}$ over small open sets $U\subset X$.
Denote by $\mathcal{N}(H^{\mathbb{C}})$ the moduli space of semistable
principal $H^{\mathbb{C}}$-bundles over $X$; the construction of
the moduli space can be found in \cite{Ra96}. It is a union of connected
components (see \cite{Ra75}) 
\[
\mathcal{N}(H^{\mathbb{C}})=\coprod_{d}\mathcal{N}_{d}(H^{\mathbb{C}})
\]
indexed by the elements $d\in\pi_{1}(H^{^{\mathbb{C}}})$ which correspond
to topological types of principal $H^{\mathbb{C}}$-bundles $E$ over
$X$. Moreover, for each $d\in\pi_{1}(H^{^{\mathbb{C}}})$, $\mathcal{N}_{d}(H^{\mathbb{C}})$
is non-empty.

If $\mathfrak{h}^{\mathbb{C}}\subset\mathfrak{g}^{\mathbb{C}}$ are
the corresponding Lie algebras, there is a (complexified) Cartan decomposition
\[
\mathfrak{g}^{\mathbb{C}}=\mathfrak{h}^{\mathbb{C}}\oplus\mathfrak{m}^{\mathbb{C}}
\]
where $\mathfrak{m}^{\mathbb{C}}$ is a complex vector space. The
restriction of the adjoint representation $\mathrm{Ad}:G^{\mathbb{C}}\to\mathrm{GL}(\mathfrak{g}^{\mathbb{C}})$
to $H^{\C}$ preserves the Cartan decomposition and induces the \emph{isotropy
representation} of $H^{\mathbb{C}}$ on $\mathfrak{m}^{\mathbb{C}}$:
\begin{equation}
\iota:H^{\mathbb{C}}\to\mathrm{GL}(\mathfrak{m}^{\mathbb{C}}).\label{isotropyrepresentation}
\end{equation}
Given a $H^{\mathbb{C}}$-bundle $E$, denote by $E(\mathfrak{m}^{\mathbb{C}})$
the vector bundle with fibres $\mathfrak{m}^{\mathbb{C}}$ associated
to $E$ via the isotropy representation, i.e., $E(\mathfrak{m}^{\mathbb{C}})=E\times_{\iota}\mathfrak{m}^{\mathbb{C}}$. 
\begin{defn}
A \emph{$G$-Higgs bundle} on a Riemann surface $X$ is a pair $(E,\varphi)$
which consists of a principal $H^{\mathbb{C}}$-bundle $E$ and a
holomorphic section $\varphi$ of the bundle $E(\mathfrak{m}^{\mathbb{C}})\otimes K$.
The section $\varphi$ is called the \emph{Higgs field}. \end{defn}
\begin{rem*}
We have the following particular cases. 
\begin{enumerate}
\item If $G$ is itself a compact group, then $\mathfrak{m}^{\mathbb{C}}=0$,
so the Higgs field is identically zero, and we recover the notion
of principal $H^{\mathbb{C}}=G^{\mathbb{C}}$-bundle. 
\item When $G$ is a complex group, then we have $H^{\mathbb{C}}=G$ and
also $\mathfrak{m}^{\mathbb{C}}=\mathfrak{g}$. So, a \emph{G-Higgs
bundle} is a pair $(E,\varphi)$, where $E$ is a $G$-bundle and
$\varphi\in H^{0}(X,E(\mathfrak{g})\otimes K)=H^{0}(X,\Ad(E)\otimes K)$. 
\item When $G$ is non-compact of Hermitian type there is an almost complex
structure on $\mathfrak{m}^{\mathbb{C}}$ defined by the adjoint action
of a special element $J$ in the center $\mathfrak{z}$ of $\mathfrak{h}$
with $J^{2}=-\mathrm{id}$. The almost complex structure splits $\mathfrak{m}^{\mathbb{C}}$
into $H^{\mathbb{C}}$-invariant $\pm i$-eigenspaces 
\[
\mathfrak{m}^{\mathbb{C}}=\mathfrak{m}^{+}\oplus\mathfrak{m}^{-}
\]
and therefore splits the bundle $E(\mathfrak{m}^{\mathbb{C}})=E(\mathfrak{m}^{+})\oplus E(\mathfrak{m}^{-})$.
Hence the Higgs field decomposes as $\varphi=(\varphi^{+},\varphi^{-})$
where 
\begin{equation}
\varphi^{+}\in H^{0}(X,E(\mathfrak{m}^{+})\otimes K),\quad\varphi^{-}\in H^{0}(X,E(\mathfrak{m}^{-})\otimes K).\label{eq:phi-hermitian-type}
\end{equation}

\end{enumerate}
\end{rem*}

The notion of $G$-Higgs bundle includes several interesting particular
cases. When $G$ is a classical Lie group, $G$-Higgs bundles can
be defined in terms of holomorphic vector bundles with additional
structure, as follows. 
\begin{example}
A $\GL(n,\mathbb{C})$-Higgs bundle on $X$ is a pair $(E,\varphi)$,
where $E$ is a rank $n$ holomorphic vector bundle over $X$ and
$\varphi\in H^{0}(\End(E)\otimes K)$ is a holomorphic endomorphism
of $E$ twisted by $K$. This is the original notion of Higgs bundle
introduced by Hitchin \cite{Hi87}. Similarly, a $\SL(n,\mathbb{C})$-Higgs
bundle is a pair $(E,\varphi)$, where $E\to X$ is a holomorphic
rank $n$ vector bundle with $\det(E)=\mathcal{O}$ and $\varphi\in H^{0}(X,\End(E)\otimes K)$
with $\tr(\varphi)=0$. 
\end{example}

\begin{example}
A $\SO(n,\mathbb{C})$-Higgs bundle is a pair $(E,\varphi)$ where
$E$ is a $\SO(n,\mathbb{C})$-bundle and $\varphi\in H^{0}(E(\mathfrak{so}(n,\mathbb{C}))\otimes K)$.
Using the standard representations of $\SO(n,\mathbb{C})$ in $\mathbb{C}^{n}$
we can associate to $E$ a holomorphic vector bundle $W$ of rank
$n$ with trivial determinant, 
\[
W=E\times_{\SO(n,\mathbb{C})}\mathbb{C}^{n},
\]
together with a non-degenerate symmetric quadratic form $Q\in H^{0}(\mathrm{S}^{2}W^{\ast})$;
we can think of $Q$ as a symmetric holomorphic isomorphism $Q:W\to W^{*}$.
The Higgs field in terms of the vector bundle $W$ is a holomorphic
section $\varphi\in H^{0}(\End(W)\otimes K)$ satisfying $Q(u,\varphi v)=-Q(\varphi u,v)$
and $\mathrm{tr}(\varphi)=0$. 
\end{example}

\begin{example}
Let $G=\SL(n,\mathbb{R})$. The Cartan decomposition of the Lie algebra
is given by 
\[
\mathfrak{sl}(n,\mathbb{R})=\mathfrak{so}(n)\oplus\mathfrak{m},
\]
where $\mathfrak{m}=\{\mbox{symmetric real matrices of trace 0}\}$.
So a $\SL(n,\mathbb{R})$-Higgs bundle is a pair $(E,\varphi)$, where
$E$ is a $\SO(n,\mathbb{C})$-bundle and $\varphi\in H^{0}(E(\mathfrak{m}^{\mathbb{C}})\otimes K)$.
Hence a $\SL(n,\mathbb{R})$-Higgs bundle can be viewed as a triple
$(W,Q,\varphi)$, where $(W,Q)$ is a holomorphic orthogonal bundle
with $\det(W)=\mathcal{O}$, and $\varphi$ is a traceless holomorphic
section of $\End(W)\otimes K$ that is symmetric with respect to $Q$,
i.e. $Q\varphi^{T}Q=\varphi$. 
\end{example}

\subsection{Moduli spaces\label{sub:Stablity-conditions}}

For the construction of moduli spaces, as usual one introduces several
notions of stability. The notions of stability, semistability and
polystability for $G$-Higgs bundles depend on a real parameter $\alpha$
and generalize the usual slope stability condition for Higgs bundles
and Ramanathan's stability condition for principal bundles. In the
present work we consider only the particular case $\alpha=0$, because
this is the relevant value for relating $G$-Higgs bundles to representations
of $\pi_{1}(X)$ via the non-abelian Hodge theorem. Thus we simply
say polystable instead of 0-polystable and likewise for stable and
semistable, and refer the reader to \cite{GGM09} for the general
definitions. 
\begin{rem}
\label{rem:stability} To a $G$-Higgs bundle, for $G\subset\GL(n,\C)$,
we can naturally associate a $\GL(n,\C)$-Higgs bundle. By this correspondence,
semistability of a $G$-Higgs bundle is equivalent to semistability
of the associated $\GL(n,\C)$-Higgs bundle. For stability the situation
is more subtle: it is possible for a stable $G$-Higgs bundle to induce
a strictly semistable $\GL(n,\C)$-Higgs bundle. 
\end{rem}

\subsection{Components of moduli spaces}

To a given $G$-Higgs bundle we can associate the topological invariant
of the underlying $H^{\mathbb{C}}$-bundle. As mentioned before, for
connected $H^{\mathbb{C}}$, topological types are well-known \cite{Ra75}
to be classified by elements of 
\[
\pi_{1}(H^{\mathbb{C}})=\pi_{1}(H)=\pi_{1}(G).
\]

\begin{defn}
For a fixed $d\in\pi_{1}(G)$, \emph{the moduli space of polystable
$G$-Higgs bundles} $\mathcal{M}_{d}(G)$ is defined to be the set
of isomorphism classes of polystable $G$-Higgs bundles $(E,\varphi)$
with $c(E)=d$. 
\end{defn}
These moduli spaces are complex algebraic varieties, due to constructions
of Schmitt \cite{Sc05,Sc08}. We have the disjoint union 
\begin{align*}
\mathcal{M}(G) & =\coprod_{d}\mathcal{M}_{d}(G).
\end{align*}
When $G$ is a complex reductive Lie group the moduli space $\mathcal{M}_{d}(G)$
is connected and non-empty, for every $d\in\pi_{1}(G)$ (see \cite{GO16}).
But the situation is very different when $G$ is a real reductive
Lie group. In this case the moduli space $\mathcal{M}_{d}(G)$ can
be a union of several connected components and can also be empty for
some $d\in\pi_{1}(G)$.

The following are three known cases of real Lie groups for which there
exists a topological type $d$ such that $\mathcal{M}_{d}(G)$ is
disconnected: 
\begin{enumerate}
\item When $G$ is a split real form, proved by Hitchin \cite{Hi92}, 
\item When $G$ is non-compact of Hermitian type, the \emph{Cayley correspondence}
\cite{BGG06} provides extra components in the moduli space for maximal
Toledo invariant (defined below). For $G=\SL(2,\R)$, this goes back
to Goldman \cite{Gol80}. 
\item When $G=\SO_{0}(p,q)$ there are, in general, extra components not
accounted for by the preceding mechanisms, see \cite{Co17,ABCGGO}. 
\end{enumerate}
In the case when $G$ is non-compact of Hermitian type one can define
an integer invariant $\tau(E,\varphi)$ called the \emph{Toledo invariant}
which is an element of the torsion free part of $\pi_{1}(H)$. This
invariant is bounded by a \emph{Milnor-Wood inequality}, beyond which
the moduli spaces are empty. In fact, if $G$ is non-compact of Hermitian
type and $(E,\varphi^{+},\varphi^{-})$ is a semistable $G$-Higgs
bundle, then the Toledo invariant $\tau=\tau(E)$ satisfies 
\begin{equation}
-\rk(\im(\varphi^{+}))(2g-2)\leq\tau\leq\rk(\im(\varphi^{-}))(2g-2),\label{eq:Milnor-Wood}
\end{equation}
(see \cite{BGR,GN}) where $\varphi^{+},\varphi^{-}$ are defined
in \eqref{eq:phi-hermitian-type}.

\begin{example}
A $\SL(2,\mathbb{R})$-Higgs bundle has the form 
\[
E=(W=L\oplus L^{*},Q=\begin{pmatrix}0 & 1\\
1 & 0
\end{pmatrix},\varphi=\left(\begin{array}{cc}
0 & \varphi^{+}\\
\varphi^{-} & 0
\end{array}\right)),
\]
where $L$ is a line bundle, $\varphi^{+}\in H^{0}(X,L^{2}\otimes K)$
and $\varphi^{-}\in H^{0}(X,L^{-2}\otimes K)$. The group $\SL(2,\R)$
is of Hermitian type and the Toledo invariant is $\tau(E)=2\deg(L)$.
The inequality (\ref{eq:Milnor-Wood}) implies $\abs{\deg(L)}\leq g-1$
and, if both $\varphi^{+}$ and $\varphi^{-}$ are non-zero, any $E$
satisfying this inequality is semistable. Moreover, if $\varphi^{+}=0$,
then $E$ is semistable if and only if $\deg(L)\geq0$ and if $\varphi^{-}=0$,
then $E$ is semistable if and only if $\deg(L)\leq0$. Thus, if the
Higgs field vanishes then $E$ is semistable if and only if $\deg(L)=0$. \end{example}
\begin{rem}
\label{rem:toldeo-deformation} It follows from (\ref{eq:Milnor-Wood})
that, if $G$ is of Hermitian type and $(E,0)$ is a semistable $G$-Higgs
bundle, then $\tau(E)=0$. 
\end{rem}
For all the real connected semisimple classical groups of Hermitian
type, namely $\SU(p,q)$, $\Sp(2n,\mathbb{R})$, $\SO^{*}(2n)$ and
$\SO_{0}(2,n)$ we have $\pi_{1}(H)\cong\Z$, except $G=\SO_{0}(2,n)$
with $n\geq3$ for which $\pi_{1}(H)\cong\Z\oplus\Z_{2}$. So in these
cases, i.e.\ excepting $\SO_{0}(2,n)$, the topological type of the
$G$-Higgs bundle is determined by the Toledo invariant.

\subsection{The $\mathbb{C}^{\ast}$-action on the moduli spaces and retraction
to the nilpotent cone}

In this subsection we show how the use of a $\mathbb{C}^{*}$-action
on the moduli space of $G$-Higgs bundles implies a deformation retraction
onto the nilpotent cone.

The moduli space of $G$-Higgs bundles $\mathcal{M}_{d}(G)$ admits
a non-trivial holomorphic $\mathbb{C}^{*}$-action \cite{Hi87,Si92}
by multiplication of the Higgs field, 
\begin{equation}
z\cdot(E,\varphi)=(E,\,z\varphi).\label{C^*-action}
\end{equation}


From the gauge theory point of view one can observe that the action
of the subgroup $S^{1}\subset\mathbb{C}^{*}$ on the moduli space
is Hamiltonian with proper moment map defined as follows 
\begin{align*}
f & :\mathcal{M}_{d}(G)\to\mathbb{R}\\
 & \quad(E,\varphi)\mapsto\norm{\varphi}^{2}:=\int_{X}|\varphi|^{2}\mathrm{vol}.
\end{align*}
When the moduli space $\mathcal{M}_{d}(G)$ is smooth, the theorem
of Frankel \cite{Fr59} implies that $f$ is a perfect Bott-Morse
function. Another consequence of the fact that $f$ is a moment map
for the Hamiltonian $S^{1}$-action is that the set of critical points
of $f$ coincides with the set of fixed points of the action. We also
recall that the sets of fixed points of the actions of $S^{1}$ and
$\C^{*}$ coincide. Let $\{\mathcal{F}_{\lambda}\}_{\lambda\in\Lambda}$
be the set of the irreducible components of the fixed point set of
the $\mathbb{C}^{*}$-action on $\mathcal{M}_{d}(G)$, with $\Lambda$
an index set.

There exists a \emph{Morse stratification} on the moduli spaces $\mathcal{M}_{d}(G)$
which coincides with ther \emph{Bia{ł}ynicki-Birula stratification},
due to results of Kirwan in \cite{Ki84}. It is defined as follows.
Let 
\[
U_{\lambda}:=\{(E,\varphi)\in\mathcal{M}_{d}(G)\,|\,\lim_{z\to0}z\cdot(E,\varphi)\in\mathcal{F}_{\lambda}\}.
\]
Then $\cup_{\lambda}U_{\lambda}$ gives a stratification of $\mathcal{M}_{d}(G)$.

One can also define the so-called \emph{downward Morse flow} of $\mathcal{F}_{\lambda}$
which, again due to the result of Kirwan, is given by the sets $D_{\lambda}:=\{(E,\varphi)\in\mathcal{M}_{d}(G)\,|\underset{z\to\infty}{\lim}z\cdot(E,\varphi)\in\mathcal{F}_{\lambda}\}$.
Using the label $0\in\Lambda$ to denote the fixed point set of $G$-Higgs
bundles with zero Higgs field, it is clear that we have $\mathcal{F}_{0}=\mathcal{N}_{d}(H^{\mathbb{C}})$.
Note that $\mathcal{M}_{d}(G)$ does not have to be smooth for the
Bia{ł}ynicki-Birula stratification to be defined.

\subsubsection{Nilpotent Cone}

\label{sec:nilpotent-cone}

Take a basis $\{\beta_{1},\cdots,\beta_{r}\}$ for the $G$-invariant
polynomials on the Lie algebra $\mathfrak{g}^{\mathbb{C}}$ (under
the adjoint action) and let $d_{i}=\deg(\beta_{i})$. Given a $G$-Higgs
bundle $(E,\varphi)$, the evaluation of $\beta_{i}$ on $\varphi$
gives a section $\beta_{i}(\varphi)\in H^{0}(X,K^{d_{i}})$. For a
fixed $d\in\pi_{1}(G)$ the (restricted) \emph{Hitchin map} is defined
to be 
\begin{align*}
\mathcal{H} & :\mathcal{M}_{d}(G)\to\bigoplus H^{0}(X,K^{d_{i}})\\
 & \quad(E,\varphi)\mapsto(\beta_{1}(\varphi),\ldots,\beta_{r}(\varphi)).
\end{align*}
For example when $G=\GL(n,\mathbb{C})$ then $\beta_{i}(\varphi)$
can be taken to be $\mathrm{tr}(\wedge^{i}\varphi)$ and $d_{i}=i$
for all $i=1,\ldots,n$. The Hitchin map is proper for any choice
of basis; see \cite{Hi87,Hi92}. A more general direct construction
(i.e.\ without passing to the complex group) of the Hitchin map for
real $G$ can be found in \cite{GPR}.

The pre-image of zero under the Hitchin map $\mathcal{H}^{-1}(0)\subset\mathcal{M}_{d}(G)$
is called the \emph{nilpotent cone}. This was defined by Laumon \cite{La88}
in the case of a complex group, and by abuse of language we use the
same name when $G$ is a real Lie group. The Hitchin map is algebraic,
so the nilpotent cone is a subscheme which is, in general, neither
reduced nor irreducible (see \cite{Hi17} for a precise analysis in
the case $G=\SL(2,\C)$). However, we shall view it as a subvariety\footnote{We shall not require varieties to be irreducible.},
i.e., we consider the associated reduced scheme. 
\begin{prop}
\cite{Ha} \label{prop:Morse-Nil} The downward Morse flow coincides
with the nilpotent cone, more precisely 
\[
\mathcal{H}^{-1}(0)=\bigcup_{\lambda\in\Lambda}\bar{D}_{\lambda}.
\]

\end{prop}
From the above proposition and the fact that $\mathcal{H}$ is proper
we can also deduce that each component of the nilpotent cone is a
projective variety. The following result generalizes the one for semisimple
complex $G$ given in \cite{BF}, with an analogous proof. 
\begin{prop}
\label{prop:deformation-nil} Let $G$ be a real reductive Lie group.
Then the nilpotent cone $\mathcal{H}^{-1}(0)$ is a deformation retraction
of the moduli space $\mathcal{M}_{d}(G)$. \end{prop}
\begin{proof}
Fixing a Hermitian metric on $X$ it induces a Hermitian metric on
$K$, and hence a inner product on each vector space $H^{0}(X,K^{d_{i}})$.
Consider the following composition map: 
\begin{align*}
\mathcal{M}_{d}(G)\stackrel{\mathcal{H}}{\to}\bigoplus_{i=1}^{r}H^{0}(X,K^{d_{i}}) & \stackrel{f}{\to}\mathbb{R}_{\geq0},\\
(s_{1},\cdots,s_{r}) & \mapsto\sum_{i=1}^{r}\norm{s_{i}}^{\frac{1}{d_{i}}}.
\end{align*}
Since both the Hitchin map $\mathcal{H}$ and $f$ are proper, the
inverse image $(f\circ\mathcal{H})^{-1}([0,\epsilon])=:U_{\epsilon}$
is a compact neighborhood of the nilpotent cone. Note that for any
real $t\geq0$ and $s_{i}\in H^{0}(X,K^{d_{i}})$ we have $\norm{t.s_{i}}=t^{d_{i}}\norm{s_{i}}$
and hence 
\begin{align}
f(ts_{1},\cdots,ts_{r}) & =t\,f(s_{1},\cdots,s_{r})\label{property-f}
\end{align}

Using the $\mathbb{C}^{*}$-action on the moduli space of $G$-Higgs
bundles $\eqref{C^*-action}$ we define the following homotopy between
the identity map of $\mathcal{M}_{d}(G)$ and a retraction onto $U_{\epsilon}$
as follows: 
\begin{align*}
\mathcal{F}: & \mbox{ }\mathcal{M}_{d}(G)\times[0,1]\to\mathcal{M}_{d}(G)\\
 & \quad(E,\varphi)\mapsto\begin{cases}
\begin{cases}
(E,t_{0}\cdot\varphi) & t\leq t_{0}:=\frac{\epsilon}{f(\mathcal{H}(E,\varphi))}\\
(E,t\cdot\varphi) & t>t_{0}
\end{cases} & \mbox{ if }f(\mathcal{H}(E,\varphi))>\epsilon\\
(E,\varphi) & \mbox{ if }f(\mathcal{H}(E,\varphi))\leq\epsilon
\end{cases}
\end{align*}
Indeed, we have 
\begin{align*}
\mathcal{F}((E,\varphi),t) & =(E,\varphi),\text{ for }(E,\varphi)\in U_{\epsilon}\\
\mathcal{F}((E,\varphi),1) & =(E,\varphi),\text{ for }(E,\varphi)\in\mathcal{M}_{d}(G).
\end{align*}
Next we prove $\mathcal{F}((E,\varphi),0)\in U_{\epsilon}$ to conclude
that $U_{\epsilon}$ is a deformation retraction of $\mathcal{M}_{d}(G)$.
Clearly if $f(\mathcal{H}(E,\varphi))\leq\epsilon$ then $\mathcal{F}((E,\varphi),0)=(E,\varphi)\in U_{\epsilon}$.
If $f(\mathcal{H}(E,\varphi))>\epsilon$ then 
\begin{align*}
f(\mathcal{H}(\mathcal{F}((E,\varphi),0)))=f(\mathcal{H}(E,t_{0}\cdot\varphi))=t_{0}\,f(\mathcal{H}(E,\varphi))=\epsilon,
\end{align*}
in the last equality we use the equality \eqref{property-f}.

The nilpotent cone is a proper subvariety of $\mathcal{M}_{d}(G)$
so it is a finite CW-complex and an absolute deformation retract (see,
for example, \cite{BCR98}). Hence, there is some open neighborhood
$U\supseteq\mathcal{H}^{-1}(0)$ such that $U$ deformation retracts
to $\mathcal{H}^{-1}(0)$. Choose $\epsilon$ small enough so that
$U_{\epsilon}\subset U$, this is possible as $\mathcal{H}$ is proper.
Therefore the composition of deformation retraction of $U$ into the
nilpotent cone and of $\mathcal{M}_{d}(G)$ into $U_{\epsilon}$ gives
a retraction of $\mathcal{M}_{d}(G)$ into the nilpotent cone. 
\end{proof}

\section{The obstructions to a deformation retraction}

\label{sec:obstruction} For every topological type $d\in\pi_{1}(H)$
there is a natural inclusion $\mathcal{N}_{d}(H^{\mathbb{C}})\subset\mathcal{M}_{d}(G)$
which comes from considering principal $H^{\mathbb{C}}$-bundles as
$G$-Higgs bundles with zero Higgs field. Thus, we have 
\[
\mathcal{N}_{d}(H^{\mathbb{C}})\subset\mathcal{H}^{-1}(0)\subset\mathcal{M}_{d}(G),
\]
and we can identify $\mathcal{N}_{d}(H^{\mathbb{C}})=\mathcal{F}_{0}$.
Thus, in order to discuss obstructions to the deformation retraction
from the moduli spaces of $G$-Higgs bundles to $\mathcal{N}_{d}(H^{\mathbb{C}})$,
by using Proposition \ref{prop:deformation-nil}, it is enough to
study the obstructions to deformation retraction from the nilpotent
cone to $\mathcal{N}_{d}(H^{\mathbb{C}})=\mathcal{F}_{0}$, which
we do next.
\begin{rem*}
In the case when $G$ is non-compact of Hermitian type, by Remark~\ref{rem:toldeo-deformation}
the right question to ask would be the deformation retraction from
$\mathcal{M}_{d}(G)$ to $\mathcal{N}_{d}(H^{\mathbb{C}})$ for trivial
topological type $d=0$. 
\end{rem*}

\subsection{Additive homology of $\mathcal{M}_{d}(G)$}

In this section we consider homology with $\C$-coefficients. The
following lemmas are of course well known but, for completeness, we
include proofs.

Recall that we do not require algebraic varieties to be irreducible.
We understand the dimension of a variety $Y$ to be the maximal dimension
of an irreducible component of $Y$. We also recall that any projective
variety has the structure of a finite CW-complex and that this can
taken to be compatible with any given subvariety \cite{BCR98,Hir75}.
Finally we recall that any irreducible projective variety $Y$ of
dimension $r$ has a non-zero fundamental class $[Y]\in H_{2r}(Y)\cong\C$
(see, e.g., \cite[II.7.6]{Ha75}) and that $H_{n}(Y)=0$ for $n>2\dim(Y)$. 
\begin{lem}
\label{lem:h-different-1} Let $Y$ be a projective variety of dimension
$r$. Then $H_{2r}(Y)\cong\C^{n}$, where $n$ is the number of irreducible
components of $Y$ of dimension $r$. \end{lem}
\begin{proof}
We prove the result by induction on the number of irreducible components
of $Y$, the case $n=1$ being the result described in the paragraph
preceding the lemma. So let $Y=Y_{1}\cup Y_{2}$, where $Y_{1}$ is
irreducible and $Y_{2}$ has $n-1$ irreducible components.

Let $r=\dim(Y)$. Since $\dim(Y_{1}\cap Y_{2})<r$ we have $H_{n}(Y_{1}\cap Y_{2})=0$
for $n>2r-2$. Thus the Mayer--Vietoris sequence for $Y=Y_{1}\cup Y_{2}$
gives 
\[
0\to H_{2r}(Y_{1})\oplus H_{2r}(Y_{2})\xrightarrow{\cong}H_{2r}(Y)\to0.
\]
Since by induction the desired result holds for $Y_{1}$ and $Y_{2}$,
the lemma follows. \end{proof}
\begin{lem}
\label{lem:h-different-2} Let $Y$ be a projective variety and suppose
that $Y=Y_{1}\cup Y_{2}$, where $Y_{i}\subsetneq Y$ is non-empty
and closed for $i=1,2$. Then $Y_{i}$ and $Y$ have non-isomorphic
homology %
{} for $i=1,2$. \end{lem}
\begin{proof}
Let $r=\dim(Y)$. If both $Y_{1}$ and $Y_{2}$ have dimension $r$,
the result is immediate from Lemma~\ref{lem:h-different-1}. It remains
to consider the case when $\dim(Y_{1})=s<r$ and $\dim(Y_{2})=r$,
say. Since clearly $Y$ and $Y_{1}$ have distinct homology we just
have to show that $Y$ and $Y_{2}$ have distinct homology. For this,
note first that we may remove any irreducible components of $Y_{1}$
which are contained in $Y_{2}$ and still have the hypotheses of the
Lemma satisfied. Then, by decomposing into irreducible components,
we see that $\dim(Y_{1}\cap Y_{2})<s$. Therefore we have $H_{n}(Y_{1}\cap Y_{2})=0$
for $n>2s-2$. Thus the Mayer--Vietoris sequence for $Y=Y_{1}\cup Y_{2}$
gives 
\[
0\to H_{2s}(Y_{1})\oplus H_{2s}(Y_{2})\xrightarrow{\cong}H_{2s}(Y)\to0.
\]
Since, by Lemma~\ref{lem:h-different-1}, $H_{2s}(Y_{1})\neq0$,
we see that $H_{2s}(Y_{2})$ and $H_{2s}(Y)$ are distinct, as desired. \end{proof}
\begin{lem}
\label{lem:wobbly-limit} Assume that there exists a component $\mathcal{F}_{\lambda}$
of the fixed locus with $\lambda\neq0$. Then we may choose $\lambda\neq0$
such that 
\[
\bar{D}_{\lambda}\cap\mathcal{F}_{0}=\left\{ \lim_{z\to0}(E,z\varphi)\suchthat(E,\varphi)\in D_{\lambda}\right\} .
\]
\end{lem}
\begin{proof}
Following Simpson \cite[S{11}]{Si94}, we may consider a $\C^{*}$-equivariant
embedding of $\mathcal{H}^{-1}(0)$ as a projective variety, where
the ambient projective space has a standard positively weighted $\C^{*}$-action
and $\mathcal{F}_{0}$ lies in the weight zero subspace. Then the
component $\mathcal{F}_{\lambda}$ with the lowest non-zero weight
of the $\C^{*}$-action satisfies the condition of the lemma. \end{proof}
\begin{prop}
\label{Prop:Deformation-Retraction} Suppose that there is a non-empty
$\mathcal{F}_{\lambda}$, for some $\lambda\neq0$. Then $\mathcal{M}_{d}(G)$
and $\mathcal{N}_{d}(H^{\mathbb{C}})$ have distinct additive singular
homology. \end{prop}
\begin{proof}
Consider the closed subspace $\bar{D}_{\lambda}\subset\mathcal{H}^{-1}(0)$.
If $\mathcal{F}_{0}$ is not contained in $\bar{D}_{\lambda}$ then
Lemma~\ref{lem:h-different-2} gives the conclusion. Otherwise Lemma~\ref{lem:wobbly-limit}
tells us that, for suitable $\lambda$, any $(E,0)\in\mathcal{F}_{0}$
is of the form $(E,0)=\lim_{z\to0}(E,z\varphi)$ with $(E,\varphi)\in D_{\lambda}$.
Now consider the $\C^{*}$-invariant subspace of $\bar{D}_{\lambda}$,
\[
\bar{D}_{\lambda}^{0}=\{(E,\varphi)\in\bar{D}_{\lambda}\suchthat\lim_{z\to0}(E,z\varphi)\in\mathcal{F}_{0}\}.
\]
By Lemma~\ref{lem:wobbly-limit}, the map $\bar{D}_{\lambda}^{0}\to\mathcal{F}_{0}$
given by $(E,\varphi)\mapsto(E,0)$ is a surjective morphism. Moreover,
it is clearly $\C^{*}$-equivariant. Hence, since the $\C^{*}$-action
is non-trivial on $\bar{D}_{\lambda}^{0}$ and trivial on $\mathcal{F}_{0}$,
we conclude that $\dim\bar{D}_{\lambda}>\dim\mathcal{F}_{0}$. Therefore
Lemma~\ref{lem:h-different-1} shows that $\mathcal{H}^{-1}(0)$
and $\mathcal{F}_{0}$ have distinct homology, as was to be shown. \end{proof}
\begin{cor}
\label{cor:retraction} Suppose that there exists a semistable $G$-Higgs
bundle $(E,\varphi)$ for which $E$ is unstable as a principal $H^{\mathbb{C}}$-bundle.
Then $\mathcal{M}_{d}(G)$ and $\mathcal{N}_{d}(H^{\mathbb{C}})$
have distinct additive singular homology. \end{cor}
\begin{proof}
Let $(E,\varphi)$ be a semistable Higgs bundle and suppose $\underset{t\to0}{\lim}(E,\,t\varphi)$
is $(E,\,0)$. Then $E$ is a semistable $H^{\mathbb{C}}$-bundle.
So, our hypothesis implies $\underset{t\to0}{\lim}(E,\,t\varphi)=(E_{0},\varphi_{0})$
with $\varphi_{0}\neq0$. Therefore $(E_{0},\varphi_{0})\in\mathcal{F}_{\lambda}$,
with $\lambda\neq0$ (as $||\varphi_{0}||\neq0$) and hence the result
follows using Proposition~\ref{Prop:Deformation-Retraction}. 
\end{proof}

\subsection{The associated Higgs bundle}

The following can be found in \cite{Hi92,GPR}. Again, let $G$ be
a real reductive Lie group with maximal compact subgroup $H$, and
$\mathfrak{g}^{\mathbb{C}}$ be the complexification of the Lie algebra
$\mathfrak{g}$ of $G$. Let $\sigma\colon\lie{g}^{\C}\to\lie{g}^{\C}$
be the corresponding ($\C$-antilinear) real structure and let $\theta\colon\lie{g}^{\C}\to\lie{g}^{\C}$
be the ($\C$-linear) Cartan involution. Consider the Cartan decomposition
\[
\mathfrak{g}^{\mathbb{C}}=\mathfrak{h}^{\mathbb{C}}\oplus\mathfrak{m}^{\mathbb{C}}
\]
into $\pm1$-eqigenspace for $\theta$.

For example, for $G=\SL(2,\R)$, the Cartan involution is $\theta\colon X\mapsto-X^{t}$
and the Cartan decomposition of $\mathfrak{sl}(2,\mathbb{C})$ under
$\theta$ is 
\[
\mathfrak{sl}(2,\mathbb{C})=\mathfrak{so}(2,\mathbb{C})\oplus\mathfrak{sym}_{0}(2,\mathbb{C})
\]
where $\mathfrak{so}(2,\C)$ denotes the trace zero complex diagonal
matrices, and $\mathfrak{sym}_{0}(2,\mathbb{C})$ the complex antidiagonal
matrices.

When $G$ is non-abelian, there is a $\sigma$ and $\theta$-equivariant
injective morphism 
\[
\rho':\mathfrak{sl}(2,\mathbb{C})\to\mathfrak{g}^{\mathbb{C}},
\]
such that $\rho'=\rho'_{+}\oplus\rho'_{-}$, where 
\begin{align}
\rho'_{+}:\mathfrak{so}(2,\mathbb{C})\to\mathfrak{h}^{\mathbb{C}},\quad\rho'_{-}:\mathfrak{sym}_{0}(2,\mathbb{C})\to\mathfrak{m}^{\mathbb{C}}.\label{morphism-sym-to-m}
\end{align}
Since $\SL(2,\mathbb{C})$ is simply-connected $\rho'$ lifts to 
\begin{align}
\rho:\SL(2,\mathbb{C}) & \to G^{\mathbb{C}}.\label{eq:morphism-SL(2,C)}
\end{align}

On the other hand, the restriction $\rho'|_{\mathfrak{sl}(2,\mathbb{R})}:\mathfrak{sl}(2,\mathbb{R})\to\mathfrak{g}$
lifts to a $\theta$-equivariant group homomorphism, still denoted
by $\rho$ 
\begin{align}
\rho:\SL(2,\mathbb{R}) & \to G\label{eq:morphism-SL(2,R)}
\end{align}
which takes $\SO(2)$ to $H$. We denote by $\rho_{+}$ the complexification
of the restriction $\rho|_{\SO(2)}$ 
\begin{align}
\rho_{+}:\SO(2,\mathbb{C}) & \to H^{\mathbb{C}}.\label{eq:rho+}
\end{align}
given an $\SL(2,\mathbb{R})$-Higgs bundle $(E^{\prime},\varphi^{\prime})$
we can construct a $G$-Higgs bundle $(E,\varphi)$ via \eqref{eq:rho+}
and \eqref{morphism-sym-to-m} in the following way: 
\begin{align}
E:=E^{\prime}\times_{\SO(2,\mathbb{C})}H^{\mathbb{C}},\quad\varphi:=\rho_{-}^{\prime}(\varphi^{\prime})\in H^{0}(X,E(\mathfrak{m}^{\C})\otimes K).\label{extended-G-Higgs}
\end{align}

More generally we have the following: let $f:G^{\prime}\to G$ be
a morphism of reductive Lie groups. This induces a morphism $f:{H'}^{\C}\to H^{\mathbb{C}}$,
still denoted by the same symbol. Given a $G^{\prime}$-Higgs bundle
$(E^{\prime},\varphi^{\prime})$ one can associate a $G$-Higgs bundle
$(E,\varphi)$, which is called the \emph{extended $G$-Higgs bundle}
via $f$, with $E:=E^{\prime}\times_{{H'}^{\C}}H^{\mathbb{C}}$. Moreover,
since $\varphi'\in H^{0}(X,E(\mathfrak{m}^{\prime^{\mathbb{C}}})\otimes K)$
we get a section of 
\[
E(\mathfrak{m}^{\mathbb{C}}):=E(\mathfrak{m}^{\prime^{\mathbb{C}}})\underset{i\circ F}{\times}\mathfrak{m}^{\mathbb{C}},
\]
where $i$ is the isotropy representation for $G^{\prime}$. Hence
$\varphi^{\prime}$ defines via the homomorphism $F:=D_{e}f:\mathfrak{g}^{\prime}\to\mathfrak{g}$,
a Higgs field $\varphi$ on $E$.

When $G$ is connected, $f$ induces a homomorphism between the fundamental
groups 
\[
f_{*}:\pi_{1}(G^{\prime})\to\pi_{1}(G)
\]
and the topological type of the associated $G$-Higgs bundle corresponds
to the image via the map $f_{*}$. We recall the following result
on polystability for the associated $G$-Higgs bundle:
\begin{prop}
\label{prop:extended-Higgs-bundle} Let $f:G^{\prime}\to G$ be a
morphism between reductive Lie groups (real or complex). Let $(E^{\prime},\varphi^{\prime})$
be a $G'$-Higgs bundle and $(E,\varphi)$ be the extended $G$-Higgs
bundle via $f$. Then, if $(E^{\prime},\varphi^{\prime})$ is semistable,
then so is $(E,\varphi)$. Thus the group homomorphism $f$ defines
a morphism 
\begin{eqnarray*}
\mathcal{M}_{d}(G^{\prime}) & \to & \mathcal{M}_{f_{*}d}(G)\\
(E^{\prime},\varphi^{\prime}) & \mapsto & (E,\varphi)
\end{eqnarray*}
\end{prop}
\begin{proof}
This follows from \cite[Corollary 5.10]{GPR}, since the stability
parameter is zero. 
\end{proof}
\begin{prop}
\label{prop:associating-unstability-principal} Let ${H'}^{\C}$ and
$H^{\C}$ be connected complex Lie groups with ${H'}^{\C}$ semisimple
and $H^{\C}$ reductive. Let $f:{H'}^{\C}\to H^{\C}$ be a morphism
with discrete kernel. Let $E^{\prime}$ be a principal $H'^{\C}$-bundle
and let $E$ be the principal $H^{\C}$-bundle obtained by extension
of structure group via $f$. If $E^{\prime}$ is unstable as a ${H'}^{\C}$-bundle,
then $E$ is unstable as a $H^{\mathbb{C}}$-bundle. \end{prop}
\begin{proof}
Since ${H'}^{\C}$ is semisimple, the unstable $H'^{\C}$-bundle $E'$
is destabilized by a reduction to a proper parabolic subgroup\footnote{The semisimple assumption on ${H'}^{\C}$ is a subtle point: for example,
a line bundle $L$ with $\deg(L)\neq0$ is $0$-unstable, however,
there is no reduction to a proper parabolic of the structure group
$\C^{*}$.}. Now, if $f$ is surjective, the result follows from \cite[Proposition~7.1]{Ra75}
--- note that one needs to ensure that the image of a proper parabolic
in ${H'}^{\C}$ is a proper parabolic in $H^{\C}$, and the hypothesis
on the kernel of $f$ achieves this. For the general case, suppose
then that $E^{\prime}$ is not stable. It follows that the principal
${H'}^{\C}/\ker(f)$-bundle obtained by extension of structure group
via ${H'}^{\C}\to{H'}^{\C}/\ker(f)$ is also unstable. Thus, since
$E$ is obtained by extension of structure group via $f\colon{H'}^{\C}/\ker(f)\to H^{\C}$,
we may assume that $f$ is injective. The result now follows from
\cite[Proposition~3.13]{GO16}; note that this is result about $G$-Higgs
bundles but of course also applies to principal bundles, viewed as
Higgs bundles with vanishing Higgs field. \end{proof}
\begin{rem}
Note that the case of surjective $f$ is equally valid for $G$-Higgs
bundles, with essentially the same proof as that of \cite[Proposition~7.1]{Ra75}.
Thus Proposition~\ref{prop:associating-unstability-principal} in
fact applies to $G$-Higgs bundles as well. We shall, however, not
need this. 
\end{rem}
The following result shows the existence of $G$-Higgs bundles $(E,\varphi)\in\mathcal{F}_{\lambda}$,
for $\lambda\neq0$ as in the hypothesis of Corollary~\ref{cor:retraction}:
\begin{prop}
\label{existence-unstable-bundle} Let $G$ be a non-abelian (real
or complex) reductive connected Lie group. Then there exists a semistable
Higgs bundle $(E,\varphi)\in\mathcal{M}(G)$ with $E$ an unstable
principal $H^{\mathbb{C}}$-bundle. \end{prop}
\begin{proof}
If $G$ is complex, consider the $\SL(2,\mathbb{C})$-Higgs bundle
\[
(K^{1/2}\oplus K^{-1/2},\varphi=\left(\begin{array}{cc}
0 & 0\\
1 & 0
\end{array}\right)).
\]
Clearly this is a stable $\SL(2,\C)$-Higgs bundle and the underlying
$\SL(2,\C)$-bundle $K^{1/2}\oplus K^{-1/2}$ is unstable. Now we
take the extended $G$-Higgs bundle via \eqref{eq:morphism-SL(2,C)}
which we denote by $(E,\varphi)$. By Proposition \ref{prop:extended-Higgs-bundle}
this is semistable and by Proposition \ref{prop:associating-unstability-principal}
$E$ is an unstable principal $G$-bundle.

For $G$ real, we can use a variation of the same idea. Consider the
basic $\SL(2,\mathbb{R})$-Higgs bundle 
\begin{equation}
(K^{1/2}\oplus K^{-1/2},Q=\left(\begin{array}{cc}
0 & 1\\
1 & 0
\end{array}\right),\varphi=\left(\begin{array}{cc}
0 & 0\\
1 & 0
\end{array}\right))\label{basic-SL(2,R)}
\end{equation}
where $1$ is the canonical section of $\Hom(K^{1/2},K^{-1/2}\otimes K)$.
Clearly this is a stable $\SL(2,\mathbb{R})$-Higgs bundle.

Let $(E,\varphi)$ be the $G$-Higgs bundle obtained from the basic
$\SL(2,\mathbb{R})$-Higgs bundle \eqref{basic-SL(2,R)} via \eqref{extended-G-Higgs}.
Then, since the diagram 
\[
\begin{CD}\SL(2,\R)@>{\rho}>>\hspace{-40pt}G@.\\
@VVV\hspace{-40pt}@VVV@.\\
\SL(2,\C)@>{\rho}>>G^{\C}\supset H^{\C}
\end{CD}
\]
commutes, we can use the argument of the previous paragraph to conclude
that the $G^{\C}$-Higgs bundle $(\tilde{E},\tilde{\varphi})$ obtained
from $(E,\varphi)$ by extension of structure group via $G\subset G^{\C}$
is a semistable $G^{\C}$-Higgs bundle, whose underlying principal
$G^{\C}$-bundle $\tilde{E}$ is unstable. Finally note that $\tilde{E}$
is obtained from $E$ by extension of structure group via the inclusion
$H^{\C}\subset G^{\C}$. Hence the principal $H^{\C}$-bundle is also
unstable (cf.\ Proposition~\ref{prop:extended-Higgs-bundle}). 
\end{proof}

\subsection{Reducibility of the nilpotent cone}

Here we deduce reducibility of the nilpotent cone when $G$ is a connected
reductive complex Lie group. Thus, in this subsection $G=H^{\C}$.
\begin{prop}
\label{top-type-complex} Let $G$ be a non-abelian connected reductive
complex Lie group. Then the topological type of the extended $G$-Higgs
bundle $(E,\varphi)$ constructed in Proposition~\ref{existence-unstable-bundle}
is zero. \end{prop}
\begin{proof}
The topological type of the basic $\SL(2,\mathbb{C})$-Higgs bundle:
\[
(K^{1/2}\oplus K^{-1/2},\varphi=\left(\begin{array}{cc}
0 & 0\\
1 & 0
\end{array}\right))
\]
which we consider in the proof of Proposition~\ref{existence-unstable-bundle}
is zero and hence the topological type of the extended $G$-Higgs
bundle is zero as well, by the induced homomorphism between the fundamental
groups $i_{*}:\thinspace\pi_{1}(\SL(2,\mathbb{C}))\to\pi_{1}(G)$
which is indeed trivial in this case. 
\end{proof}
For the proof of the next result we shall need the notion of very
stable $G$-bundles which we recall from \cite{La88,BR94}: A principal
$G$-bundle $P$ is said to be \emph{very stable} if $H^{0}(X,\mathrm{ad}P\otimes K)$
does not contain any non-zero nilpotent Higgs field.
\begin{prop}
\label{rem:nilpotent-element} Let $G$ be a non-abelian connected
reductive complex Lie group. Then, the nilpotent cone contains a component
which does not belong to $\mathcal{N}_{0}(G)$. \end{prop}
\begin{proof}
It follows from Proposition~\ref{existence-unstable-bundle} and
Proposition~\ref{top-type-complex} that there exists a semistable
$G$-Higgs bundle $(E,\varphi)$ of trivial topological type for which
$E$ is unstable as a principal $H^{\mathbb{C}}$-bundle. This implies
that there is some $\lambda\neq0$ such that $F_{\lambda}$ is non
empty, see proof of Corollary \ref{cor:retraction}. And on the other
hand, using the existence of very stable $G$-bundles result, \cite[Corollary 5.6]{BR94},
we can conclude that $\mathcal{F}_{0}$ is not contained in $\bar{D}_{\lambda}$
and hence the result follows. \end{proof}
\begin{thm}
\label{thm:reducibility-nilpotent} Let $G$ be a non-abelian connected
reductive complex Lie group. Then the nilpotent cone in the moduli
space $\mathcal{M}_{0}(G)$ of $G$-Higgs bundles of trivial topological
type is not irreducible. \end{thm}
\begin{proof}
It is immediate from Proposition~\ref{rem:nilpotent-element}.\end{proof}
\begin{rem}
The above result was shown in \cite{BF} in the semisimple complex
case. Our result extends this to the complex reductive case. Since
in the case of real reductive $G$ we do not have existence result
of very stable $G$-bundles we could not conclude reducibility of
the nilpotent cone for this case. 
\end{rem}

\subsection{Non retracting and topological type}

By putting together our previous result here we prove that the moduli
space of $G$-Higgs bundles does not deformation retract onto the
moduli space of principal bundles. Since we want to study the obstructions
to a deformation retraction from $\mathcal{M}(G)$ to $\mathcal{N}(H^{\mathbb{C}})$,
we should consider separately each topological type, and here we consider
trivial topological type. Thus, in order to apply Corollary~\ref{cor:retraction},
we should look for a semistable $G$-Higgs bundle $(E,\varphi)$ in
$\mathcal{M}_{0}(G)$ for which $E$ is unstable as a principal $H^{\mathbb{C}}$-bundle.
The following result shows that Proposition~\ref{existence-unstable-bundle}
gives a topologically trivial $G$-Higgs bundle with unstable underlying
$H^{\mathbb{C}}$-bundle.
\begin{prop}
\label{pro:topological-type} We have the following: 
\begin{itemize}
\item[(i)] Let $G$ be a non-abelian connected simple real Lie group of Hermitian
non-tube type. Then the topological type of the extended $G$-Higgs
bundle $(E,\varphi)$ constructed in Proposition~\ref{existence-unstable-bundle}
is zero. 
\item[(ii)] Let $G$ be a non-abelian connected reductive real Lie group of non-Hermitian
type. Then there is a polystable $G$-Higgs bundle $(E,\varphi)$
of trivial topological type such that $E$ is an unstable $H^{\mathbb{C}}$-bundle. 
\end{itemize}
\end{prop}
\begin{proof}
$Part\,(i)$ follows from \cite[Proposition 7.1, Proposition 7.2]{GPR}.
To prove $Part\,(ii)$, let $\tilde{G}$ be the universal cover of
$G$ and hence we have a surjective Lie group homomorphism 
\[
p:\tilde{G}\to G
\]
such that $\ker(p)$ lies in the center of $\tilde{G}$. By Proposition~\ref{existence-unstable-bundle}
we obtain a polystable $\tilde{G}$-Higgs bundle $(\tilde{E},\tilde{\varphi})$
with unstable $H^{\mathbb{C}}$-bundle and since $\tilde{G}$ is simply-connected
the topological type of $\tilde{E}$ is trivial. Therefore, by using
Proposition~\ref{prop:associating-unstability-principal} and Proposition~\ref{prop:extended-Higgs-bundle}
the extended $G$-Higgs bundle via the covering map is the desired
$G$-Higgs bundle. \end{proof}
\begin{rem*}
When $G$ is a connected simple real Lie group of Hermitian tube type
then the topological type of the extended $G$-Higgs bundle $(E,\varphi)$
as in Proposition~\ref{existence-unstable-bundle} is maximal, see
\cite[Proposition 7.2]{GPR}. Since we are studying the obstructions
to a deformation retraction from the moduli space of polystable $G$-Higgs
bundles of trivial topological type $\mathcal{M}_{0}(G)$ to $\mathcal{N}_{0}(H^{\mathbb{C}})$
we exclude this case in the above Proposition. 
\end{rem*}
Finally putting our results together we obtain the following theorem.
Note that the moduli spaces are generally singular. 
\begin{thm}
\label{thm:homotopy-type} Let $G$ be a non-abelian (real or complex)
connected reductive Lie group of non-Hermitian type or connected simple
real Lie group of Hermitian non-tube type. Then the moduli space of
semistable principal $H^{\C}$-bundles of trivial topological type
$\mathcal{N}_{0}(H)$ is not a deformation retraction of the moduli
space $\mathcal{M}_{0}(G)$ of semistable $G$-Higgs bundles of trivial
topological type. \end{thm}
\begin{proof}
Combine Corollary~\ref{cor:retraction}, Proposition~\ref{existence-unstable-bundle},
Proposition~\ref{top-type-complex} and Proposition~\ref{pro:topological-type}. \end{proof}


\begin{thebibliography}{ABCGGO18}
\bibitem[ABCGGO18]{ABCGGO} M. Aparicio-Arroyo, S.~Bradlow, B.~Collier,
O.~Garcia-Prada, P.~Gothen, and A.~Oliveira, \emph{Exotic components
of $\mathrm{SO}(p,q)$ surface group representations, and their Higgs
bundle avatars}, C.~R.~Acad.\ Sci.\ Paris, Ser.~I \textbf{356} (2018), 666--673.

\bibitem[BF11]{BF} I.~Biswas and C.~Florentino, \emph{The topology
of moduli spaces of group representations: The case of compact surface},
Bull. Sci. Math. \textbf{135} (2011), 395--399.

\bibitem[BCR98]{BCR98} J. Bochnak, M. Coste and M. F. Roy, Real algebraic
geometry, Springer-Verlag (1998).

\bibitem[BGG06]{BGG06} S.~B. Bradlow, O.~Garc{í}a-Prada, and
P.~B. Gothen, \emph{Maximal surface group representations in isometry
groups of classical Hermitian symmetric spaces}, Geom. Dedicata \textbf{122}
(2006), 185--213.

\bibitem[BGR]{BGR} O.~Biquard, O.~Garc{í}a-Prada, and R.~Rubio,
\emph{Higgs bundles, Toledo invariant and the Cayley correspondence},
Journal of Topology \textbf{10} (2017), 795--826.

\bibitem[BR94]{BR94} I. Biswas and S. Ramanan, \emph{An infinitesimal
study of the moduli of Hitchin pairs}, J. London Math. Soc. (2) \textbf{49}
(1994), 219--231.

\bibitem[Co17]{Co17} B. Collier, \emph{$\mathrm{SO}(n,n+1)$-surface
group representations and their Higgs bundles}, \texttt{arXiv:1710.01287
{[}math.DG{]}}, 2017.

\bibitem[Co88]{Co88} K.~Corlette, \emph{Flat ${G}$-bundles with
canonical metrics}, J. Differential Geom. \textbf{28} (1988), 361--382.

\bibitem[Do87]{Do87} S.~K. Donaldson, \emph{Twisted harmonic maps
and the self-duality equations}, Proc. London Math. Soc. (3) \textbf{55}
(1987), 127--131.

\bibitem[Fr59]{Fr59} T. Frankel, \emph{Fixed points and torsion on
Kähler manifolds}, Ann. of Math. (2) \textbf{70} (1959), 1--8.

\bibitem[FL09]{FL} C. Florentino, S. Lawton, \emph{The topology of
moduli spaces of free group representations}, Math. Ann. \textbf{345}
(2009), 453--489.

\bibitem[Gol80]{Gol80} W. Goldman, Discontinuous groups and the Euler
class, Ph.D. thesis, University of California, Berkeley, 1980.

\bibitem[Go14]{Go14} P.~B.~Gothen, \emph{Representations of surface
groups and {H}iggs bundles}, Moduli Spaces, London Mathematical
Society Lecture Note Series (No. 411), 151--178, Cambridge University
Press 2014.

\bibitem[GGM09]{GGM09} O.~Garc{í}a-Prada, P.~B. Gothen, and I.~Mundet~i
Riera, \emph{The Hitchin-Kobayashi correspondence, {H}iggs pairs
and surface group representations}, preprint, 2012, \texttt{arXiv:0909.4487v3
math.AG}.

\bibitem[GN]{GN} P.~B.~Gothen, A.~Nozad, \emph{Birationality of
moduli spaces of twisted $\U(p,q)$-Higgs bundles}, Rev.\ Mat.\
Complut. \textbf{30} (2017) 91--128.

\bibitem[GO16]{GO16} O.~Garc{í}a-Prada and André Oliveira, \emph{Connectedness
of Higgs bundle moduli for complex reductive Lie groups}, Asian J.
Math. \textbf{21} (2017), 791--810.

\bibitem[GPR15]{GPR} O.~Garc{í}a-Prada, A.~Peón-Nieto, S.~Ramanan,
\emph{Higgs bundles for real groups and the Hitchin-Kostant-Rallis
section}, Trans. Amer. Math. Soc. \textbf{370} (2018), 2907-–2953.

\bibitem[Ha]{Ha} T. Hausel, \emph{Geometry of the moduli space of
Higgs bundles}, Ph.D. thesis, University of Cambridge, 1998. \texttt{arXiv:math/0107040}.

\bibitem[Ha75]{Ha75} R. Hartshorne, \emph{On the Rham cohomology
of algebraic varieties}, Inst. Hautes {É}tudes Sci. Publ. Math.
\textbf{45} (1975), 5--99.

\bibitem[Hir75]{Hir75} H.~Hironaka, \emph{Triangulations of algebraic
sets}, Algebraic geometry (Proc. Sympos. Pure Math., Vol. 29, Humboldt
State Univ., Arcata, Calif., 1974), Amer. Math. Soc., Providence,
R.I., 1975, pp.~165--185.

\bibitem[Hi87]{Hi87} N.~J.~Hitchin, \emph{The self-duality equations
on a Riemann surface}, Proc. London Math. Soc. (3) \textbf{55} (1987),
59--126.

\bibitem[Hi92]{Hi92} N.~J.~Hitchin, \emph{Lie groups and Teichmüller
space}, Topology \textbf{31} (1992), 449--473.

\bibitem[Hi17]{Hi17} N.~J.~Hitchin, \emph{Critical loci for Higgs
bundles}, preprint 2017, \texttt{arXiv:1712.09928 {[}math.DG{]}}.

\bibitem[Ki84]{Ki84}F. C. Kirwan, Cohomology of quotients in symplectic
and algebraic geometry. Mathematical Notes, 31. Princeton University
Press, Princeton, NJ, 1984.

\bibitem[La88]{La88} G.~Laumon, \emph{Un analogue global du cone
nilpotent}, Duke Math. J. \textbf{57} (1988), 647--671.

\bibitem[Ra75]{Ra75} A. Ramanathan, \emph{Stable principal bundles
on a compact Riemann surface}, Math. Ann. \textbf{213} (1975) 129--152.

\bibitem[Ra96]{Ra96} A.~Ramanathan, \emph{Moduli for principal bundles
over algebraic curves}, Proc. Indian Acad. Sci. Math. Sci. \textbf{106}
(1996) 301--328, and 421--449.

\bibitem[Sc05]{Sc05}A. Schmitt, \emph{Moduli for decorated tuples
of sheaves and representation spaces for quivers}, Proc. Indian Acad.
Sci. Math. Sci. \textbf{115} (2005), 15–-49.

\bibitem[Sc08]{Sc08}A. Schmitt, Geometric invariant theory and decorated
principal bundles. European Mathematical Society, 2008.

\bibitem[Si88]{Si88} C.~T.~Simpson, \emph{Constructing variations
of {H}odge structure using {Y}ang-{M}ills theory and applications
to uniformization}, J. Amer. Math. Soc. \textbf{1} (1988), 867--918.

\bibitem[Si92]{Si92} C.~T.~Simpson, \emph{Higgs bundles and local
systems}, Inst. Hautes {É}tudes Sci. Publ. Math. \textbf{75} (1992),
5--95.

\bibitem[Si94]{Si94} C.T. Simpson, \emph{Moduli of representations
of the fundamental group of a smooth projective variety}. II, Publ.
Math. Inst. Hautes Études Sci. \textbf{80} (1994) 5--79.\end{thebibliography}
\end{document}